\documentclass[]{interact}

\usepackage[pdf]{lfmath}

\usepackage{epstopdf}
\usepackage[caption=false]{subfig}

\usepackage[numbers,sort&compress]{natbib}
\bibpunct[, ]{[}{]}{,}{n}{,}{,}
\renewcommand\bibfont{\fontsize{10}{12}\selectfont}
\makeatletter
\def\NAT@def@citea{\def\@citea{\NAT@separator}}
\makeatother

\theoremstyle{plain}
\newtheorem{theorem}{Theorem}[section]

\newtheorem{proposition}[theorem]{Proposition}

\theoremstyle{definition}

\newtheorem{example}[theorem]{Example}

\theoremstyle{remark}

\newcommand{\fx}{f(x)}
\newcommand{\gx}{g(x)}
\newcommand{\hx}{h(x)}
\newcommand{\sqfx}{\sqrt{f(x)}}
\newcommand{\sqgx}{\sqrt{g(x)}}
\newcommand{\sqhx}{\sqrt{h(x)}}
\renewcommand{\Re}{\text{Re}}
\renewcommand{\Im}{\text{Im}}
\newenvironment{solution}{\begin{proof}[Solution]}{\end{proof}}
\newcommand{\Gkioulekas}{\mbox{}Gkioulekas\mbox{}}
\let\emptyset\varnothing

\allowdisplaybreaks
\begin{document}

\articletype{}
\title{Solving parametric radical equations with depth 2 rigorously using the restriction set method}

\author{
\name{Eleftherios Gkioulekas\textsuperscript{a}\thanks{CONTACT Eleftherios Gkioulekas. Email: drlf@hushmail.com}}
\affil{\textsuperscript{a}School of Mathematical and Statistical Sciences, University of Texas Rio Grande Valley, 1201 West University Drive, Edinburg, TX 78539-2999}
}

\maketitle

\begin{abstract}
We review the history and previous literature on radical equations and present the rigorous solution theory for radical equations of depth 2, continuing a previous study of radical equations of depth 1. Radical equations of depth 2 are equations where the unknown variable appears under at least one square root and where two steps are needed to eliminate all radicals appearing in the equation. We state and prove theorems for all three equation forms with depth 2 that give the solution set of all real-valued solutions. The theorems are shown via the restriction set method that uses inequality restrictions to decide whether to accept or reject candidate solutions. We distinguish between formal solutions that satisfy the original equation in a formal sense, where we allow some radicals to evaluate to imaginary numbers during verification, and strong solutions, where all radicals evaluate to real numbers during verification. Our theorems explicitly identify the set of all formal solutions and the set of all strong solutions for each equation form. The theory underlying radical equations with depth 2 is richer and more interesting than the theory governing radical equations with depth 1, and some aspects of the theory are not intuitively obvious. It is illustrated with examples of parametric radical equations.
\end{abstract}


\begin{amscode}
00A05, 01A55, 97H30
\end{amscode}

\section{Introduction}

In this article  we present the rigorous theory for solving radical equations of depth 2, and illustrate its application with examples of parametric radical equations. We define radical equations to be equations in which the unknown variable appears at least once under a square root. By depth 2, we mean that two steps are needed to eliminate all radicals and reduce the equation to a polynomial or rational equation. We say that an equation is parametric when it has known variables (parameters) and the goal is to find the unknown variables in terms of the parameters, such that the equation is satisfied. We limit our scope to equations where all radicals are square roots and to seeking real-valued solutions. The article is a continuation of  previous work \cite{article:Gkioulekas:radicals-one} where a rigorous theory was presented for solving several forms of radical equations of depth 1. The goal is to be able to find all real-valued solutions of a radical equation and to be able to eliminate all of the extraneous solutions, without having to inconveniently verify them by substituting each candidate solution to the original equation. We also highlight and address an ambiguity in properly defining the set of real-valued solutions of a radical equation. The main advantage of the proposed approach is that it makes it practical to study the general solution of parametric radical equations in terms of one or more parameters. 

The research literature on radical equations is not extensive, and was  reviewed in a previous article~\cite{article:Gkioulekas:radicals-one}. Nagase~\cite{article:Nagase:1987} presented the general theory for solving the equation $\sqrt{ax+b} = cx+d$, using a technique by Bompart~\cite{article:Bompart:1982} and Roberti~\cite{article:Roberti:1984}. The general solution of the  equation $\sqrt{ax+b}+\sqrt{cx+d}=A$ was given by Huff \& Barrow~\cite{article:Barrow:1982}. The reverse problem of constructing radical equations following several forms from the desired solutions is also non-trivial and was briefly discussed by Beach~\cite{article:Beach:1952} and Schwartz, Moulton, \& O'Hara~\cite{article:O'Hara:1987}. The solution technique for handling one of the depth 1 radical equation forms was discussed informally by Gurevich~\cite{article:Gurevich:2003}, however an informal discussion of several depth 1 and depth 2 forms was taught to me earlier by my high school instructor \cite{notes:Pistofidis:1989,notes:Pistofidis:1990}, and, as explained in the following,  a nascent  precursor of some of the solution techniques that I have learned from him can be traced back to Fischer \& Schwatt~\cite{book:Schwatt:1898}. An overview of the history of the broader problem of handling the extraneous solutions, when solving rational or radical equations, throughout the 19th and 20th century was given by Manning~\cite{article:Manning:1970}, and we summarize the particulars about radical equations in the following. 

Radical equations first appeared in mathematics textbooks around 1860. At the time, there was no concern about extraneous solutions because the notation $\sqrt{a}$ had a multi-valued interpretation where it could be equal to either zero of the polynomial $p(x) = x^2-a$. Oliver, Wait, \& Jones~\cite{book:Jones:1887} introduced the notation $\sqrt[-]{a}$ and $\sqrt[+]{a}$ to distinguish between the negative and positive zero of $p(x) = x^2-a$, while retaining a multivalued interpretation for $\sqrt{a}$. As we see on page 215, Oliver et al.~\cite{book:Jones:1887} considered the statement $x^2=2$ equivalent to $x=\sqrt{2}$ which, in turn, was seen as equivalent to $x=\sqrt[+]{2} \lor x=\sqrt[-]{2}$. This is not consistent with the modern use of the radical notation, however a multi-valued interpretation of the radical sign and fractional powers is still used today in the context of complex analysis.  Under this multivalued definition, the  statement  $\sqrt{4x+1}=x-5$, for example, was viewed as equivalent to the following statement, under the modern single-valued definition  of the radical sign:
\begin{equation}
\sqrt{4x+1}=x-5 \lor -\sqrt{4x+1}=x+5
\label{eq:multivalued-radicals-eq-one}
\end{equation} 
As a result, any  solution that is an extraneous solution, for one of the two equations in the disjunction given by Eq.~\eqref{eq:multivalued-radicals-eq-one}, will satisfy the other equation and vice versa. Consequently, under the old multivalued definition, the problem of extraneous solutions simply disappears. The same principle applies to equations of depth 2. For example, the equation $\sqrt{x^2-a^2}+\sqrt{x^2+a^2}=bx$, considered in Example~\ref{ex:sum-two-roots-equal-to-function-one},  would have been understood, under the multivalued definition, as equivalent to the following statement, under the modern single-valued definition:
\begin{align*}
\sqrt{x^2-a^2}&+\sqrt{x^2+a^2}=bx \lor
\sqrt{x^2-a^2}-\sqrt{x^2+a^2}=bx \nonumber \\
&\lor
-\sqrt{x^2-a^2}+\sqrt{x^2+a^2}=bx \lor
-\sqrt{x^2-a^2}-\sqrt{x^2+a^2}=bx 
\end{align*}
Again, any extraneous solution obtained from any one of the four equations in the disjunction above will satisfy at least one of the other three equations. Oliver et al.~\cite{book:Jones:1887} expended a substantial amount of effort to present a very rigorous and interesting  theory of radicals, from the bottom up, under the multivalued definition. The reader may also be  interested in the very captivating  biographical memoir  \cite{article:Hill:1896} about the character and career of Professor James Edward Oliver (1829--1895), available from the National Academy of Sciences, for giving insight, not only to a very interesting teacher-scholar, but also for a taste of  American academia during the 19th century. 

The problem of extraneous solutions in radical equations was noticed when mathematicians began attributing a single-valued definition to the radical sign. By 1898, Fischer \& Schwatt~\cite{book:Schwatt:1898} had introduced the term \emph{principal root} for the positive root, and used the notation $\sqrt{a}$ to represent the positive root thus revealing the problem of extraneous solutions in radical equations. On pages 552-554, Fischer \& Schwatt~\cite{book:Schwatt:1898}  showed solved examples of several radical equations, using their principal root definition of the radical notation, consistently with its modern meaning. It is interesting to note that Fischer \& Schwatt~\cite{book:Schwatt:1898} tried using simple contradiction arguments to eliminate extraneous solutions, whenever they could, but without developing the more systematic methodology for these arguments,   that we find in Pistofidis~\cite{notes:Pistofidis:1989,notes:Pistofidis:1990}, Gurevich~\cite{article:Gurevich:2003}, \Gkioulekas~\cite{article:Gkioulekas:radicals-one}, and the present article. Manning~\cite{article:Manning:1970}  noted that a rigorous approach to handling extraneous solutions was abandoned by many textbooks during the 20th century where the broader problem of extraneous solutions was discussed superficially,  with the sole recommendation that they can be eliminated by verification against the original equation. Taylor~\cite{article:Taylor:1910}, Hegeman~\cite{article:Hegeman:1922}, Bruce~\cite{article:Bruce:1931}, and Allendoerfer~\cite{article:Allendoerfer:1966}, tried to revive interest in a more rigorous approach to handling the broader problem of extraneous solutions.  Hegeman~\cite{article:Hegeman:1922}, in particular, noted the following critique of what has now become the standard textbook approach for teaching radical equations:
\begin{quotation}
  ``The student naturally wants to be shown why a root obtained by a
  process which he has been taught to consider correct is not a root
  at all. He is usually told that it is an extraneous root. This
  explains nothing and he is just as puzzled as before''.
\end{quotation}
He proposes that a more transparent pedagogy for solving radical equations is moving all terms of the equation to the left-hand side and then multiplying both sides with rationalizing factors to progressively eliminate the radicals. In the end, the student still needs to verify all solutions against the original equation, but the extraneous solutions can be easily explained as zeroes of the rationalizing factors introduced in the process. The solution techniques presented in Ref.~\cite{article:Gkioulekas:radicals-one} and the present article go one step further and eliminate the need to verify the solutions against the original equation.

Juxtaposed against the standard approach to radical equations and extraneous solutions there are several reasons motivating the current study. First, when solving parametric radical equations in which the equation coefficients depend on one or several parameters, the solutions found may be genuine solutions for some values of the parameters and become extraneous solutions for other values of the parameters. It is not practical to handle such situations by substituting the solutions back to the original equation. Second,  in some cases (see Example~\ref{ex:root-differences-example}) we can rule out the existence of solutions before even solving the equation in the first place. Last, but not least, as was first noted in my previous article \cite{article:Gkioulekas:radicals-one}, the concept of a real-valued solution to a radical equation needs to be carefully defined, and there is a choice between two possible definitions: A \emph{strong solution} is defined to be a real-valued solution that verifies the original equation without encountering any negative numbers under any radical sign. A \emph{formal solution} is defined to be a real-valued solution that verifies the original equation, where, in doing so, we allow radicals to evaluate to imaginary numbers. To the best of my knowledge, this distinction was not previously discussed in the literature. 

In Ref.~\cite{article:Gkioulekas:radicals-one}, we illustrated this distinction between strong solutions and formal solutions by noting that the equation $\sqrt{1-3x} = \sqrt{x-7}$ is satisfied by $x=2$, but when we substitute $x=2$, the two sides of the equation evaluate to $\sqrt{1-3x} = \sqrt{x-7} = i\sqrt{5}$. According to the definitions just given, $x=2$ is a formal solution but it is not a strong solution. Deciding whether the solution $x=2$ should be accepted or rejected depends on the broader context, which informs whether we need the set of all strong solutions or the set of all formal solutions. For example, limiting ourselves to strong solutions becomes necessary when the problem at hand  is to determine the points of intersection between the graphs of two real-valued functions, or, equivalently, the points of intersection between the graph of a real-valued function and the $x$-axis.  

In my previous article  \cite{article:Gkioulekas:radicals-one}, we gave the rigorous solution procedures for the depth 1 radical equations following the forms given by
\begin{align}
&\sqfx = \sqgx, \nonumber \\
&\sqfx = \gx, \nonumber \\
&\sqrt{f_1 (x)}+\sqrt{f_2 (x)}+\cdots +\sqrt{f_n (x)} = 0. \label{eq:radical-form-three}
\end{align}
In this article we consider the following radical equations with depth 2: 
\begin{align}
&\sqfx+\sqgx = \hx, \label{eq:radical-form-four} \\
&\sqfx+\sqgx = \sqhx, \label{eq:radical-form-five} \\
&\sqfx - \sqgx = \hx. \label{eq:radical-form-six}
\end{align}
The functions $f,g,h$ are  either polynomials or rational functions. Note that certain additional forms reduce to the aforementioned forms given by Eq.~\eqref{eq:radical-form-four}, Eq.~\eqref{eq:radical-form-five}, and Eq.~\eqref{eq:radical-form-six}. For example, $\sqfx-\sqgx = \sqhx$ is equivalent to $\sqgx+\sqhx = \sqfx$, which follows the form of Eq.~\eqref{eq:radical-form-five}. Likewise, $\sqfx-\sqgx=-\sqhx$ reduces to $\sqfx+\sqhx = \sqgx$, following the form of Eq.~\eqref{eq:radical-form-five}, and equations of the form $\sqfx+\sqgx = -\sqhx$ reduce to $\sqfx+\sqgx+\sqhx =0$, which follows the form of Eq.~\eqref{eq:radical-form-three}. It is therefore sufficient to consider solution techniques only for equations following the forms of Eq.~\eqref{eq:radical-form-four}, Eq.~\eqref{eq:radical-form-five}, and Eq.~\eqref{eq:radical-form-six} in order to take care of all possibilities with depth 2. 

Our main results are Proposition~\ref{prop:sum-two-roots-equal-to-function}, Proposition~\ref{prop:sum-of-two-roots-equal-root}, and Proposition~\ref{prop:root-difference0-prop-five} which justify the proposed solution procedures for radical equations that follow the forms of Eq.~\eqref{eq:radical-form-four}, Eq.~\eqref{eq:radical-form-five}, Eq.~\eqref{eq:radical-form-six}. The propositions explicitly identify the set of all strong solutions as well as the wider set of all formal solutions. They are illustrated with Example~\ref{ex:sum-two-roots-equal-to-function-one}, Example~\ref{ex:sum-of-two-roots-equal-root}, and Example~\ref{ex:root-differences-example}, where we demonstrate the power of the propositions on general types of parametric radical equations. For the sake of brevity, we employ a direct application of the propositions on the given examples. However, we also describe more informal solution procedures, justified by the propositions, that can be easily used in a more pedagogical teaching context, for simpler problems that are not parameter dependent.

 We find that for radical equations that follow the form of Eq.~\eqref{eq:radical-form-four}, the set of all formal solutions is always equal to the set of all strong solutions, whereas this will not necessarily be true for radical equations that follow the form of Eq.~\eqref{eq:radical-form-five} or Eq.~\eqref{eq:radical-form-six}. Unlike the case of radical equations with depth 1, our results for the set of all formal solutions given by Proposition~\ref{prop:sum-of-two-roots-equal-root} and Proposition~\ref{prop:root-difference0-prop-five}  are not intuitively obvious and a careful proof is needed to justify them. Furthermore, the constraints needed to eliminate the extraneous solutions are not immediately apparent from the initial form of the original radical equation. 

This article is organized as follows. Section 2 discusses the solution of equations following the form of Eq.~\eqref{eq:radical-form-four}. Equations following the form of Eq.~\eqref{eq:radical-form-five}  are discussed in Section 3 and equations following the form of Eq.~\eqref{eq:radical-form-six} are discussed in Section 4. The article is concluded with Section 5, where we also discuss how the main results can be incorporated in the undergraduate-level teaching of mathematics. In some of the solved examples it is necessary to determine whether the zeroes of some quadratic equation lie within particular intervals. A practical technique for doing so is given in Appendix~\ref{app:numbers-relative-to-zeroes}. 


\section{Sum of two radicals equal to a function}

For equations following the form $\sqfx+\sqgx = \hx$ we will show, with Proposition~\ref{prop:sum-two-roots-equal-to-function},  that the set of all formal solutions coincides with the set of all strong solutions. Furthermore, Proposition~\ref{prop:sum-two-roots-equal-to-function} suggests the following informal solution technique:
\begin{enumerate}
\item We require that $\hx\geq 0 \ifonlyif\cdots\ifonlyif x\in A_1$. 
\item We raise both sides to power 2 and obtain: 
\begin{align*}
\sqfx+\sqgx &= \hx
\ifonlyif
(\sqfx+\sqgx)^2 = [\hx]^2 \\
&\ifonlyif
\fx+2\sqrt{\fx\gx}+\gx = [\hx]^2  \\
&\ifonlyif
2\sqrt{\fx\gx} = [\hx]^2-\fx-\gx.
\end{align*}
\item Before raising both sides to power 2 again, we introduce the requirement 
\begin{equation*}
[\hx]^2-\fx-\gx \geq 0 \ifonlyif\cdots\ifonlyif x\in A_2.
\end{equation*}
\item We raise to power 2 again and obtain the set $S_0$ of all candidate solutions:
\begin{align*}
2\sqrt{\fx\gx} &= [\hx]^2-\fx-\gx \\
&\ifonlyif
4\fx\gx = ([\hx]^2-\fx-\gx)^2 \\
&\ifonlyif
\cdots\ifonlyif
x\in S_0.
\end{align*}
\item We accept all solutions in $S_0$ that belong to both $A_1$ and $A_2$. The solution set $S$ for all strong solutions is given by $S=S_0 \cap A_1 \cap A_2$. This is also the set of all formal solutions. 
\end{enumerate}

The notation $\ifonlyif$ represents logical equivalence (``if and only if''), with its use being essential when writing the solution of equations rigorously,  and the ellipsis $\cdots$ represent problem-dependent algebraic steps. In this procedure, steps (3) and (4) are justified by Proposition~\ref{prop:one-radical}, however steps (1) and (2) require the broader argument of Proposition~\ref{prop:sum-two-roots-equal-to-function}.  In Example~\ref{ex:sum-two-roots-equal-to-function-one}, we use a direct application of Proposition~\ref{prop:sum-two-roots-equal-to-function} to solve the equation $\sqrt{x^2-a^2}+\sqrt{x^2+a^2}=bx$ for all $a,b\in\bbR$ under the assumption that $a\neq 0$. It should be noted that it can be tricky finding interesting examples for this particular radical equation form, as many of the simplest cases tend to reduce to cubic or quartic equations, or even polynomial equations of higher order.

Proposition~\ref{prop:one-radical},  corresponding to the solution of equations that follow the form $\sqfx = \gx$, was shown as Proposition 4.2 of Ref.~\cite{article:Gkioulekas:radicals-one}.
\begin{proposition}
Consider the equation $\sqfx=\gx$ with $f: A\to\bbR$ and $g: B\to\bbR$ polynomial or rational functions with $A\subseteq\bbR$ and $B\subseteq\bbR$. The set $S_1$ of all strong solutions and the set $S_2$ of all formal solutions to the equation are given by 
\begin{align*}
S_1 &= S_2 = S_0 \cap A_1, \\
S_0 &= \{ x\in A\cap B \;|\; \fx = [\gx]^2 \}, \\
A_1 &= \{ x\in A\cap B \;|\; \gx\geq 0 \}.
\end{align*}
\label{prop:one-radical}
\end{proposition}
\noindent
Equivalently we may write: 
\begin{equation*}
\sqfx = \gx \ifonlyif \systwo{\fx=[\gx]^2}{\gx\geq 0.}
\end{equation*}
Here we use the braces notation as a shorthand for the logical ``and'' Boolean operation (i.e. conjunction) and we shall continue to do so throughout this article. It is worth noting that the assumption that the functions $f$ and $g$ be polynomials or rational functions is needed solely to justify the claims made about the set of all formal solutions. If we would like to determine only the set of all strong solutions, then this assumption can be removed, making it possible to use Proposition 1 to solve, by recursive application,  radical equations  with higher depths. The main challenge in the development of a rigorous theory for radical equations rests with determining  the formal solutions set. We now state Proposition~\ref{prop:sum-two-roots-equal-to-function} and present the corresponding Example~\ref{ex:sum-two-roots-equal-to-function-one}:
\begin{proposition}
Consider the equation $\sqfx+\sqgx=\hx$ with $f: A\to\bbR$ and $g: B\to\bbR$ and $h: C\to\bbR$ polynomial or rational functions with $A\subseteq\bbR$ and $B\subseteq\bbR$ and $C\subseteq\bbR$. The set of all strong solutions $S_1$ and the set of all formal solutions $S_2$ are both given by 
\begin{align*}
S_1 &= S_2 = S_0 \cap A_1 \cap A_2, \\
S_0 &= \{ x\in A\cap B\cap C \;|\; 4\fx\gx = [(\hx)^2-\fx-\gx]^2 \}, \\
A_1 &= \{ x\in A\cap B \cap C \;|\; \hx\geq 0\}, \\
A_2 &= \{ x\in A\cap B \cap C \;|\;  (\hx)^2-\fx-\gx \geq 0 \}.
\end{align*}
\label{prop:sum-two-roots-equal-to-function}
\end{proposition}

\begin{proof}
We claim that all formal solutions to the original equation satisfy $\fx\geq 0$ and $\gx\geq 0$ which implies that they are also going to be strong solutions. To show the claim, let us assume that there is a formal solution $x\in A \cap B \cap C$ to the equation such that $\fx <0 \lor \gx <0$. We note that, by definition, $\hx\in\bbR$ and distinguish between the following cases:

\noindent
\emph{Case 1:} Assume that $\fx <0 \land \gx\geq 0$. Then, we have 
\begin{align*}
\systwo{\sqgx\in [0,+\infty)}{\exists a\in (0,+\infty): \sqfx = ai.}
\end{align*}
Choose an $a\in (0,+\infty)$ such that $\sqfx = ai$ and note that
\begin{align*}
 \Im (\hx) &= \Im (\sqfx+\sqgx) && [\text{via the equation}] \\
&= \Im (\sqfx) && [\sqgx\in [0, +\infty)] \\
&= \Im (ai) = a\neq 0 \implies  \hx\not\in\bbR.
\end{align*}
This is a contradiction, since we know that $\hx\in\bbR$. 

\noindent
\emph{Case 2:} Assume that $\fx\geq 0 \land \gx <0$. This case also leads to a contradiction using the same argument as in Case 1.

\noindent
\emph{Case 3:} Assume that $\fx <0 \land \gx <0$. Then, we have 
\begin{equation*}
\exists a,b \in (0,+\infty): (\sqfx = ai \land \sqgx =bi).
\end{equation*}
Choose $a,b \in (0,+\infty)$ such that $\sqfx = ai$ and $\sqgx = bi $, and note that
\begin{align*}
\Im (\hx) &= \Im (\sqfx+\sqgx) && [\text{via the equation}]  \\
&= \Im (ai+bi) = a+b > a && [b\in (0,+\infty)] \\
&> 0 && [a\in (0, +\infty)] \\
&\implies \hx\not\in\bbR.
\end{align*}
This is a contradiction, since $\hx\in\bbR$.

Since all cases above yield a contradiction, we conclude that there are no formal solutions that violate the condition $\fx\geq 0 \land \gx\geq 0$, and this proves the claim. 

To solve the equation, let $x\in A\cap B\cap C$ such that $\fx\geq 0 \land \gx\geq 0$. We distinguish between the following cases. 

\noindent
\emph{Case 1:} Assume that $\hx <0$. Then, we have 
\begin{align*}
\systwo{\fx\geq 0}{\gx\geq 0}
&\implies
\systwo{\sqfx\geq 0}{\sqgx\geq 0}
\implies
\sqfx+\sqgx \geq 0 \\
&\implies
\sqfx+\sqgx \neq \hx.
\end{align*}
It follows that the equation has no formal solutions with $\hx <0$. 

\noindent
\emph{Case 2:} Assume that $\hx\geq 0$. Then, we can write,
\begin{align}
 \sqfx+\sqgx=\hx 
&\ifonlyif
 (\sqfx+\sqgx)^2=[\hx ]^2 \nonumber \\
&\ifonlyif
\fx+2\sqfx\sqgx+\gx=[\hx]^2 \nonumber \\
&\ifonlyif
2\sqrt{\fx\gx} = [\hx]^2-\fx-\gx. \label{eq:sum-two-roots-equal-to-function-two}
\end{align}
The first equivalence is justified in both directions because we have established that both sides of the equation $\sqfx+\sqgx=\hx$ are positive or zero. Using Proposition~\ref{prop:one-radical}, it follows from Eq.~\eqref{eq:sum-two-roots-equal-to-function-two}  that,
\begin{align*}
\text{Eq.~\eqref{eq:sum-two-roots-equal-to-function-two}}
\ifonlyif
\systwo{4\fx\gx = ((\hx )^2-\fx-\gx)^2}{(\hx)^2-\fx-\gx \geq 0}.
\end{align*}

From the above argument, we conclude that 
\begin{align*}
\sqfx+\sqgx=\hx &\ifonlyif
\systhree{4\fx\gx = ((\hx)^2-\fx-\gx)^2}{(\hx)^2-\fx-\gx \geq 0}{\fx\geq 0 \land \gx\geq 0 \land \hx\geq 0} \\
&\ifonlyif x\in S_1 = S_0 \cap A_1 \cap A_2,
\end{align*}
which concludes our proof.
\end{proof}

\begin{example}
The equation $\sqrt{x^2-a^2}+\sqrt{x^2+a^2}=bx$ with $a\neq 0$ has two candidate solutions:
\begin{equation*}
x_1 = -\left[ \frac{4a^4}{(2-b)(2+b)b^2}\right]^{1/4} \quad\text{ and } x_2 = +\left[ \frac{4a^4}{(2-b)(2+b)b^2}\right]^{1/4}
\end{equation*}
which are accepted or rejected as follows:
\begin{enumerate}
\item If $b\in (-\infty,-2]\cup (-\sqrt{2},\sqrt{2})\cup [2,+\infty)$, then both $x_1$ and $x_2$ are rejected
\item If $b\in (-2, -\sqrt{2}]$, then $x_1$ is accepted as a strong solution and $x_2$ is rejected. 
\item If $b\in [\sqrt{2}, 2)$, then $x_2$ is accepted as a strong solution and $x_1$ is rejected
\end{enumerate}
\label{ex:sum-two-roots-equal-to-function-one}
\end{example}

\begin{solution}
We apply Proposition~\ref{prop:sum-two-roots-equal-to-function} using $\fx = x^2-a^2$, $\gx = x^2+a^2$, and $\hx = bx$, all of which are defined on $\bbR$. The restriction set $A_1$ is obtained from the requirement 
\begin{equation*}
x\in A_1 \ifonlyif \hx\geq 0 \ifonlyif bx\geq 0,
\end{equation*}
and it follows that
\begin{equation*}
A_1=\casethree{[0,+\infty)}{b>0}{\bbR}{b=0}{(-\infty,0]}{b<0.}
\end{equation*}
For the restriction set $A_2$, we note that
\begin{equation*}
[\hx]^2-\fx-\gx = (bx)^2-(x^2-a^2)-(x^2+a^2) = (b^2-2)x^2,
\end{equation*}
and therefore,
\begin{align*}
x\in A_2 &\ifonlyif [\hx]^2-\fx-\gx \geq 0 \ifonlyif  (b^2-2)x^2 \geq 0 \\
&\ifonlyif (b-\sqrt{2})(b+\sqrt{2})x^2\geq 0 \\
&\ifonlyif (b-\sqrt{2})(b+\sqrt{2})\geq 0 \lor x=0 \\
&\ifonlyif b\in (-\infty,-\sqrt{2}]\cup [\sqrt{2},+\infty) \lor x=0,
\end{align*}
and the result is that when $b\in (-\infty,-\sqrt{2}]\cup [\sqrt{2},+\infty)$, then there is no additional restriction on $x$, otherwise $x$ has to satisfy $x=0$. It follows that
\begin{equation*}
A_2=\casetwo{\bbR}{b\in (-\infty,-\sqrt{2}]\cup [\sqrt{2},+\infty)}{\{0\}}{b\in (-\sqrt{2}, \sqrt{2}).}
\end{equation*}
To find the set $S$ of all formal/strong solutions, we let $x\in A_1\cap A_2$ be given, and, via Proposition~\ref{prop:sum-two-roots-equal-to-function}, it follows that
\begin{align}
x\in S &\ifonlyif 4\fx\gx = [(\hx)^2-\fx-\gx]^2 \nonumber \\
&\ifonlyif 4(x^2-a^2)(x^2+a^2)=[(b^2-2)x^2]^2 \nonumber \\
&\ifonlyif (2-b)(2+b)b^2 x^4 = 4a^4. \label{eq:sum-two-roots-equal-to-function-eq-two}
\end{align}
At this point we need to know whether the coefficient of $x^4$ is positive, negative, or zero, consequently we distinguish between the following cases:

\noindent
\emph{Case 1:} Assume that $b\in\{-2,0,2\}$. Then, for all $x\in A_1\cap A_2$, we have
\begin{equation*}
\text{Eq.}~\eqref{eq:sum-two-roots-equal-to-function-eq-two} \ifonlyif 0x^4 = 4a^4 \ifonlyif x\in S_0,
\end{equation*}
with $S_0 = \emptyset$. It follows that the set of all strong and all formal solutions is given by: $S = S_0\cap A_1\cap A_2 = \emptyset\cap A_1\cap A_2 = \emptyset$. 

\noindent
\emph{Case 2:} Assume that $b\in (-\infty, -2)\cup (2,+\infty)$. It follows that $(2-b)(2+b)b^2 <0$, consequently for all $x\in A_1\cap A_2$, we have:
\begin{equation*}
\text{Eq.}~\eqref{eq:sum-two-roots-equal-to-function-eq-two}\ifonlyif x^4 = \frac{4a^4}{(2-b)(2+b)b^2} < 0 \ifonlyif x\in S_0,
\end{equation*}
with $S_0 = \emptyset$. Similarly to Case 1,  it follows that the set of all strong and all formal solutions is given by $S = \emptyset$. 

\noindent
\emph{Case 3:} Assume that $b\in (-2,-\sqrt{2}]\cup [\sqrt{2},2)$. Then we have $(2-b)(2+b)b^2 > 0$, and therefore for all $x\in A_1\cap A_2$, it follows that 
\begin{equation*}
\text{Eq.}~\eqref{eq:sum-two-roots-equal-to-function-eq-two}  \ifonlyif x^4 = \frac{4a^4}{(2-b)(2+b)b^2} > 0 \ifonlyif x\in S_0,
\end{equation*}
with $S_0$ given by
\begin{equation*}
S_0 = \left\{-\left[ \frac{4a^4}{(2-b)(2+b)b^2}\right]^{1/4}, +\left[\frac{4a^4}{(2-b)(2+b)b^2} \right]^{1/4}\right\}.
\end{equation*}
Since the sets $A_1$ and $A_2$ are given by
\begin{equation*}
A_1 = \casetwo{[0,+\infty)}{b\in [\sqrt{2},2)}{(-\infty, 0]}{b\in (-2, -\sqrt{2}],}
\end{equation*}
and $A_2 = \bbR$ if $b\in (-2,-\sqrt{2}]\cup [\sqrt{2},2)$, we conclude that the set of all strong and all formal solutions is given by
\begin{align*}
S &= S_0\cap A_1\cap A_2 = S_0 \cap A_1\cap \bbR = S_0\cap A_1 \\
&= \casetwo{\ds\left\{+\left[ \frac{4a^4}{(2-b)(2+b)b^2}\right]^{1/4}\right\}}{b\in [\sqrt{2}, 2)}{\ds\left\{-\left[ \frac{4a^4}{(2-b)(2+b)b^2}\right]^{1/4}\right\}}{b\in (-2,-\sqrt{2}].}
\end{align*}

\noindent
\emph{Case 4:} Assume that $b\in (-\sqrt{2},0)\cup (0,\sqrt{2})$. Then,  we have $(2-b)(2+b)b^2 > 0$ and using the same argument that we have used for case 3 we find that
\begin{equation*}
\text{Eq.}~\eqref{eq:sum-two-roots-equal-to-function-eq-two}  \ifonlyif x\in S_0,
\end{equation*}
with $S_0$ given by
\begin{equation*}
S_0 = \left\{-\left[ \frac{4a^4}{(2-b)(2+b)b^2}\right]^{1/4}, +\left[\frac{4a^4}{(2-b)(2+b)b^2} \right]^{1/4}\right\}.
\end{equation*}
Since $A_2 = \{0\}$, and the assumption $a\neq 0$ implies that $S_0\cap A_2 = \emptyset$, it follows that the set of all strong and all formal solutions is given by
\begin{equation*}
S = S_0\cap A_1\cap A_2 = \emptyset\cap A_1 = \emptyset.
\end{equation*}

From the previous arguments given for each of the previous four cases, we conclude that the solution set $S$ of all strong and all formal solutions is given by
\begin{equation*}
S = \casethree{\emptyset}{b\in (-\infty,-2]\cup (-\sqrt{2},\sqrt{2})\cup [2,+\infty)}{\ds\left\{+\left[ \frac{4a^4}{(2-b)(2+b)b^2}\right]^{1/4}\right\}}{b\in [\sqrt{2}, 2)}{\ds\left\{-\left[ \frac{4a^4}{(2-b)(2+b)b^2}\right]^{1/4}\right\}}{b\in (-2, -\sqrt{2}],}
\end{equation*}
with the assumption that $a\neq 0$, which proves the claim.
\end{solution}

\section{Sum of two square roots equal to another square root}

Now, we turn our attention to equations that follow the form $\sqfx+\sqgx = \sqhx$.  Unlike the preceding case, this form tends to yield problems that are easier to solve, from an algebraic point of view. However, the sets of formal and strong solutions do not necessarily coincide. The informal solution technique for finding only the strong solutions is fairly intuitive and proceeds as follows. 
\begin{enumerate}
\item First, we require that all expressions under a square root be positive or zero: 
\begin{equation*}
\fx\geq 0 \land \gx\geq 0 \land \hx\geq 0 
\ifonlyif\cdots\ifonlyif
x\in A_1.
\end{equation*}
\item Then we raise both sides of the equation to the power $2$, which reads: 
\begin{align}
\sqfx &+\sqgx = \sqhx
\ifonlyif
(\sqfx+\sqgx)^2 = \hx \nonumber \\
&\ifonlyif
\fx+2\sqrt{\fx\gx}+\gx = \hx \nonumber \\
&\ifonlyif
2\sqrt{\fx\gx} = \hx-\fx-\gx.
\label{eq:sum-of-two-roots-equal-root-one}
\end{align}
\item Now, we introduce the additional requirement that 
\begin{equation*}
\hx-\fx-\gx \geq 0
\ifonlyif\cdots\ifonlyif
x\in A_2.
\end{equation*}
\item Finally we raise both sides to power $2$ again to eliminate the remaining root: 
\begin{align*}
\text{Eq.}~\eqref{eq:sum-of-two-roots-equal-root-one}
&\ifonlyif
4\fx\gx = (\hx-\fx-\gx)^2 \\
&\ifonlyif\cdots\ifonlyif
x\in S_0.
\end{align*}
\item We accept all solutions of $S_0$ that also belong to $A_1$ and $A_2$. Thus, the set of all strong solutions is given by $S=S_0 \cap A_1 \cap A_2$. 
\end{enumerate}

The requirement in step (1) is due to seeking only strong solutions. It, in turn, justifies step (2), since both sides of the equation will be positive or zero for all $x\in A_1$. Then, the steps (3) and (4) are justified via Proposition~\ref{prop:one-radical}. The new result, provided by Proposition~\ref{prop:sum-of-two-roots-equal-root}, which is not intuitively obvious,  is that if we want to find \emph{all formal solutions}, it is  necessary and sufficient to also accept all solutions that satisfy the restriction 
\begin{equation}
\systwo{\fx\leq 0 \land \gx\leq 0 \land \hx\leq 0}{\hx-\fx-\gx \leq 0.}
\label{eq:sum-of-two-roots-equal-root-formal-sols-condition}
\end{equation}
These additional solutions, if they exist, will not necessarily be strong solutions, but they will indeed satisfy the original equation in a formal sense. This claim is justified via  Proposition~\ref{prop:sum-of-two-roots-equal-root}. It is worth noting that it is a non-trivial result that these additional solutions, if they exist, emerge from the same solution set $S_0$ as the strong solutions. That is because repeating the solution process, under the assumption of Eq.~\eqref{eq:sum-of-two-roots-equal-root-formal-sols-condition}, results in a different equation at step (2) in the above procedure, but reverts back to the same equation at step (4). Example~\ref{ex:sum-of-two-roots-equal-root} uses Proposition~\ref{prop:sum-of-two-roots-equal-root} to derive the general solution of the equation $\sqrt{x+a}+\sqrt{x-a}=\sqrt{x+b}$ with $a>0$ and $b\in\bbR$, where we show that, although the equation reduces to a quadratic equation that always has two candidate solutions, at least one solution is always rejected. With a simple change of variables, the result of Example ~\ref{ex:sum-of-two-roots-equal-root} can be used to handle the more general form $\sqrt{x+a}+\sqrt{x+b}=\sqrt{x+c}$.

\begin{proposition}
Consider the equation $\sqfx+\sqgx = \sqhx$ with $f: A\to\bbR$ and $g: B\to\bbR$ and $h: C\to\bbR$ polynomial or rational functions with $A \subseteq\bbR$ and $B \subseteq\bbR$ and $C \subseteq\bbR$. Then the set of all strong solutions $S_1$ is given by 
\begin{align*}
S_1 &= S_0 \cap A_1 \cap A_2, \\
S_0 &= \{ x\in A\cap B\cap C \;|\;  4\fx\gx = [\hx-\fx-\gx]^2\},  \\
A_1 &= \{ x\in A\cap B\cap C \;|\;  \fx\geq 0 \land \gx\geq 0 \land \hx\geq 0 \}, \\
A_2 &= \{ x\in A\cap B\cap C \;|\;  \hx-\fx-\gx \geq 0 \},
\end{align*}
 and the set $S_2$ of all formal solutions is given by
\begin{align*}
S_2 &= (S_0 \cap A_1 \cap A_2)\cup (S_0 \cap A_3 \cap A_4), \\
A_3 &= \{ x\in A\cap B\cap C \;|\;  \fx\leq 0 \land \gx\leq 0 \land \hx\leq 0 \}, \\ 
A_4 &= \{ x\in A\cap B\cap C \;|\;  \hx-\fx-\gx \leq 0 \}.
\end{align*}
\label{prop:sum-of-two-roots-equal-root}
\end{proposition}

\begin{proof}
\emph{Strong solutions:} We begin with determining first the set $S_1$ of all strong solutions. We note that strong solutions have to satisfy $\fx\geq 0 \land \gx\geq 0 \land \hx\geq 0$. It follows that both sides of the original equation are positive or zero, and we can therefore justify raising both sides to power 2, as was discussed in \cite{article:Gkioulekas:radicals-one}:
\begin{align}
\sqfx+\sqgx = \sqhx
&\ifonlyif
(\sqfx+\sqgx)^2 = \hx \nonumber \\
&\ifonlyif
\fx+2\sqfx\sqgx+\gx = \hx \nonumber \\
&\ifonlyif
2\sqrt{\fx\gx} = \hx-\fx-\gx. \label{eq:sum-of-two-roots-equal-root-eq-one}
\end{align}
 From Proposition~\ref{prop:one-radical}, it follows that 
\begin{align*}
\text{Eq.~\eqref{eq:sum-of-two-roots-equal-root-eq-one}}
&\ifonlyif
\systwo{4\fx\gx = [\hx-\fx-\gx]^2}{\hx-\fx-\gx \geq 0} \\
&\ifonlyif x\in S_1 = S_0 \cap A_1 \cap A_2.
\end{align*}
We conclude that all strong solutions are given by $S_1 = S_0 \cap A_1 \cap A_2$.

\emph{Formal solutions:} Now we investigate whether there exist any additional formal solutions. Let $x\in A\cap B \cap C$ be given and assume it is a formal solution but not a strong solution of the original equation.  We distinguish between the following cases: 

\noindent
\emph{Case 1:} Assume that $\fx >0$ and $\gx <0$. Then 
\begin{align*}
\Re (\sqhx) &= \Re (\sqfx+\sqgx) && [\text{via the equation}] \\
&= \Re (\sqfx) && [\text{via } \gx <0] \\
&\neq 0, && [\text{via } \fx >0]
\end{align*}
and 
\begin{align*}
\Im (\sqhx) &= \Im (\sqfx + \sqgx) 
&& [\text{via the equation}] \\
&= \Im (\sqgx) && [\text{via } \fx >0] \\
&\neq 0. && [\text{via } \gx <0]
\end{align*}
This is a contradiction because\footnote{Here, the notation $\lxor$ represents the ``exclusive or'' Boolean operation.} 
\begin{equation*}
\hx\in\bbR 
\implies 
\hx <0 \lxor \hx\geq 0
\implies 
\Re (\sqhx) = 0 \lxor \Im (\sqhx) = 0.
\end{equation*}

\noindent
\emph{Case 2:} Assume that $\fx <0$ and $\gx >0$. This case also leads to the same contradiction as in the preceding case. 

\noindent
\emph{Case 3:} Assume that $\fx\leq 0$ and $\gx\leq 0$ and $\hx >0$. Then, we have
\begin{align*}
 \Re (\sqhx) &= \Re (\sqfx+\sqgx) && [\text{via the equation}] \\
&= \Re (\sqfx) && [\text{via } \gx\leq 0] \\
&= 0, && [\text{via } \fx\leq 0]
\end{align*}
which is a contradiction because $\hx >0 \implies \Re (\sqhx)\neq 0$. 

From the above argument it follows that all formal solutions that are not strong solutions have to satisfy the condition $\fx\leq 0 \land \gx\leq 0 \land \hx\leq 0$. We now proceed to determine the set of all such additional formal solutions. Let $x\in A\cap B\cap C$ be given such that $\fx\leq 0 \land \gx\leq 0 \land \hx\leq 0$. Then it follows that 
\begin{align}
\sqfx&+\sqgx=\sqhx
\ifonlyif
i\sqrt{-\fx}+i\sqrt{-\gx} = i\sqrt{-\hx} \nonumber \\
&\ifonlyif
\sqrt{-\fx}+\sqrt{-\gx} = \sqrt{-\hx} \label{eq:sum-of-two-roots-equal-root-eq-two-square-step} \\
&\ifonlyif
(\sqrt{-\fx}+\sqrt{-\gx})^2  = -\hx \nonumber \\
&\ifonlyif
-\fx+2\sqrt{-\fx}\sqrt{-\gx}+[-\gx] = -\hx \nonumber \\
&\ifonlyif
2\sqrt{\fx\gx} = \fx+\gx-\hx.
\label{eq:sum-of-two-roots-equal-root-eq-two}
\end{align}
The equivalence after Eq.~\eqref{eq:sum-of-two-roots-equal-root-eq-two-square-step} is justified because, by hypothesis, both sides of the equation are positive or zero. From Proposition~\ref{prop:one-radical}, it follows that 
\begin{align*}
\text{Eq.~\eqref{eq:sum-of-two-roots-equal-root-eq-two}}
&\ifonlyif
\systwo{4\fx\gx = [\fx+\gx-\hx]^2}{\fx+\gx-\hx \geq 0}  \\
&\ifonlyif
\systwo{4\fx\gx = [\hx-\fx-\gx]^2}{\hx-\fx-\gx \leq 0} \\
&\ifonlyif
x \in S_0 \cap A_3 \cap A_4.
\end{align*}
We conclude that the set of all formal solutions includes the additional solutions in $S_0 \cap A_3 \cap A_4$ and it is therefore given by $S_2 = (S_0 \cap A_1 \cap A_2) \cup (S_0 \cap A_3 \cap A_4)$. 
\end{proof}

\begin{example}
The equation $\sqrt{x+a}+\sqrt{x-a}=\sqrt{x+b}$ with $a\in (0,+\infty)$ and $b\in\bbR$ has two candidate solutions
\begin{equation*}
x_1 = \frac{-b-2\sqrt{b^2+3a^2}}{3} 
\text{ and } 
x_2 = \frac{-b+2\sqrt{b^2+3a^2}}{3},
\end{equation*}
which are accepted or rejected as follows:
\begin{enumerate}
\item If $b\in (-\infty,-a]$, then $x_1$ is a formal but not a strong solution and $x_2$ is rejected.
\item If $b\in (-a,a)$, then $x_1$ and $x_2$ are both rejected.
\item If $b\in [a,+\infty)$, then $x_2$ is a strong solution and $x_1$ is rejected. 
\end{enumerate}
\label{ex:sum-of-two-roots-equal-root}
\end{example}

\begin{solution}
We apply Proposition~\ref{prop:sum-of-two-roots-equal-root} using $\fx = x+a$, $\gx=x-a$, and $\hx=x+b$, all defined on $\bbR$. Noting that $\hx-\fx-\gx=b-x$, the corresponding restriction sets are given by
\begin{align*}
A_1 &= \{ x\in \bbR \;|\;  \fx\geq 0 \land \gx\geq 0 \land \hx\geq 0 \} \\
&= \{ x\in \bbR \;|\;  x+a\geq 0 \land x-a\geq 0 \land x+b\geq 0 \} \\
&= [\max\{-a,a,-b\},+\infty) = [\max\{a,-b\},+\infty), \\
A_2 &= \{ x\in\bbR \;|\;  \hx-\fx-\gx \geq 0 \} = \{ x\in\bbR \;|\; b-x \geq 0 \} = (-\infty,b], \\
A_3 &= \{ x\in \bbR \;|\;  \fx\leq 0 \land \gx\leq 0 \land \hx\leq 0 \} \\
&= \{ x\in \bbR \;|\;  x+a\leq 0 \land x-a\leq 0 \land x+b\leq 0 \} \\
&= (-\infty, \min\{-a,a,-b\}] = (-\infty,\min\{-a,-b\}] = (-\infty,-\max\{a,b\}], \\
A_4 &= \{ x\in\bbR \;|\;  \hx-\fx-\gx \leq 0 \} = \{ x\in\bbR \;|\; b-x \leq 0 \} = [b, +\infty).
\end{align*}
For the set $S_0$ of candidate solutions, we have
\begin{align*}
x\in S_0 &\ifonlyif 4\fx\gx = [\hx-\fx-\gx]^2 \\
&\ifonlyif 4(x+a)(x-a)=(b-x)^2 \\
&\ifonlyif 3x^2+2bx-4a^2-b^2=0 \\
&\ifonlyif x= \frac{-b-2\sqrt{b^2+3a^2}}{3} \equiv x_1 \lor x = \frac{-b+2\sqrt{b^2+3a^2}}{3} \equiv x_2.
\end{align*}
To determine whether $x_1, x_2$ can be accepted or rejected, we need to determine their position relative to the numbers $-a,a,b$, so we define $\gf (x) = 3x^2+2bx-4a^2-b^2$ and note that $\gf (-a)=-(a+b)^2$, and $\gf(a)=-(a-b)^2$, and $\gf (b) = 4(b^2-a^2)$. Then, we distinguish between the following cases:

\noindent
\emph{Case 1:} Assume that $b\in (-\infty, -a)$. Then $b<-a<0<a<-b$, and therefore, the restriction set for the strong solutions is given by
\begin{equation*}
A_1\cap A_2 = [\max\{a,-b\},+\infty)\cap (-\infty,b] = [-b,+\infty)\cap (-\infty,b] = \emptyset
\end{equation*}
where we have used $a<-b$ and $b<0$, and the restriction set for any additional formal solutions is given by
\begin{equation*}
A_3\cap A_4 = (-\infty,-\max\{a,b\}]\cap [b,+\infty) = (-\infty,-a]\cap [b,+\infty) = [b,-a]
\end{equation*}
where we have used $b<a$ and $b<-a$. It follows that $x_1, x_2$ are not strong solutions, but they may be formal solutions. To determine that, we note that from $b<-a<0$, we have $\gf (b)=4(b^2-a^2)>0$ and $\gf (-a)=-(a+b)^2 <0$, and, using the argument of Appendix~\ref{app:numbers-relative-to-zeroes}, we obtain $b<x_1<-a<x_2$, and therefore $x_1\in A_3\cap A_4$ and $x_2\not\in A_3\cap A_4$. We conclude that $x_1$ is accepted as a formal solution and $x_2$ is rejected. 

\noindent
\emph{Case 2:} Assume that $b=-a$. Then, the corresponding restriction sets are given by
\begin{align*}
A_1\cap A_2 &= [\max\{a,-b\},+\infty)\cap (-\infty, b] = [\max\{a,a\},+\infty)\cap (-\infty,-a] \\
&= [a,+\infty)\cap (-\infty,-a] = \emptyset, \\
A_3\cap A_4 &= (-\infty,-\max\{a,b\}]\cap [b,+\infty) = (-\infty,-\max\{a,-a\}]\cap [-a,+\infty) \\ 
&= (-\infty,-a]\cap [-a,+\infty)=\{-a\},
\end{align*}
and it follows that both $x_1$ and $x_2$ are not strong solutions. Furthermore. they simplify to $x_1=-a$ and $x_2=5a/3$, therefore $x_1\in A_3\cap A_4$ and is thus a formal solution, whereas $x_2\not\in A_3\cap A_4$, so it is rejected.

\noindent
\emph{Case 3:} Assume that $b\in (-a,a)$. Then $-a<-b<a$ and $-a<b<a$, so it follows that the corresponding restriction sets are given by
\begin{equation*}
A_1\cap A_2 = [\max\{a,-b\},+\infty)\cap (-\infty, b] = [a,+\infty)\cap (-\infty, b] = \emptyset,
\end{equation*}
where we have used $-b<a$ and $b<a$, and likewise
\begin{equation*}
A_3\cap A_4 = (-\infty,-\max\{a,b\}]\cap [b,+\infty) = (-\infty,-a]\cap [b,+\infty) = \emptyset,
\end{equation*}
where we have used $b<a$ and $-a<b$. We conclude that $x_1, x_2$ are both rejected as strong solutions and as formal solutions. 

\noindent
\emph{Case 4:} Assume that $b=a$. Then, the corresponding restriction sets are given by
\begin{align*}
A_1\cap A_2 &= [\max\{a,-b\},+\infty)\cap (-\infty, b] = [\max\{a,-a\},+\infty)\cap (-\infty, a] \\
&= [a,+\infty)\cap (-\infty, a]=\{a\}, \\
A_3\cap A_4 &= (-\infty,-\max\{a,b\}]\cap [b,+\infty) = (-\infty,-a]\cap [a,+\infty)=\emptyset,
\end{align*}
noting that since $a>0$, it follows that $-a<a$. The candidate solutions simplify to $x_1=-5a/3$ and $x_2=a$, consequently $x_1$ is rejected and $x_2$ is accepted as a strong solution. 

\noindent
\emph{Case 5:} Assume that $b\in (a,+\infty)$. Then, we have $-b<-a<0<a<b$, and it follows that the corresponding restriction sets are given by
\begin{equation*}
A_1\cap A_2 = [\max\{a,-b\},+\infty)\cap (-\infty, b] = [a,+\infty)\cap (-\infty, b] = [a,b],
\end{equation*}
where we have used $-b<a$ and $a<b$, and likewise
\begin{equation*}
A_3\cap A_4 = (-\infty,-\max\{a,b\}]\cap [b, +\infty) = (-\infty,-b]\cap [b,+\infty) = \emptyset,
\end{equation*}
where we have used $a<b$ and $b>a>0$. Since $\gf (a)=-(a-b)^2<0$, via $a<b$, and $\gf (b)=4(b^2-a^2)>0$, via $b>a>0$, using the argument of Appendix~\ref{app:numbers-relative-to-zeroes}, we obtain $x_1<a<x_2<b$, which implies that $x_1\not\in A_1\cap A_2$ and $x_2\in A_1\cap A_2$. Noting that there are no additional formal solutions that may not be strong solutions, since $A_3\cap A_4=\emptyset$, we conclude that $x_1$ is rejected and $x_2$ is accepted as a strong solution. 
\end{solution}

\section{Difference of square roots equal to a function}

We conclude by considering radical equations that follow the form $\sqfx-\sqgx = \hx$ where we face the following two difficulties: first, the set of all formal solutions is not necessarily equal to the set of all strong solutions; second, these equations tend to simplify to polynomials of high order from which the extraction of all possible solutions can turn out to be challenging, requiring the use of procedures for solving cubic or quartic equations.

First of all, we begin with the observation that squaring both sides of the original equation form is not a viable solution strategy because, due to the difference of the two square roots on the left-hand-side, there is no practical way to ensure that both sides of the equation will be positive for all formal solutions. Consequently, in order to find all strong solutions, we follow the following strategy instead:
\begin{enumerate}
\item We impose the requirement that  $\fx$ be positive or zero, in order to limit ourselves to strong solutions. Note that the condition $\gx\geq 0$ is not needed by Proposition~\ref{prop:root-difference0-prop-five}
\begin{equation*}
\fx\geq 0 \ifonlyif\cdots\ifonlyif x\in A_1.
\end{equation*}
\item Using Proposition~\ref{prop:one-radical}, we write: 
\begin{align*}
\sqfx-\sqgx = \hx
&\ifonlyif
\sqgx = \sqfx - \hx \\
&\ifonlyif
\systwo{\gx = (\sqfx -\hx)^2}{\sqfx-\hx \geq 0,}
\end{align*}
which results in the additional constraint 
\begin{equation*}
\sqfx-\hx \geq 0 \ifonlyif\cdots\ifonlyif x\in A_2.
\end{equation*}
\item Removing the remaining square root is justified via Proposition~\ref{prop:root-difference0-prop-four}, and the corresponding procedure reads: 
\begin{align*}
\gx &= (\sqfx-\hx)^2 \\
&\ifonlyif
\gx = \fx-2\hx\sqfx+(\hx)^2 \\
&\ifonlyif
2\hx\sqfx = \fx-\gx+(\hx)^2  \\
&\ifonlyif
\systwo{4(\hx)^2 \fx = [\fx-\gx+(\hx)^2]^2}{\hx [\fx-\gx+(\hx)^2] \geq 0.}
\end{align*}
This results in one final constraint: 
\begin{equation*}
\hx[\fx-\gx+(\hx)^2] \geq 0 
\ifonlyif\cdots\ifonlyif
x\in A_3.
\end{equation*}
\item The set of all possible solutions $S_0$ is obtained by solving, 
\begin{equation*}
4(\hx)^2 \fx = [\fx-\gx+(\hx)^2]^2
\ifonlyif\cdots\ifonlyif
x\in S_0,
\end{equation*}
 and the set of all strong solutions $S_1$ is given by $S_1 = S_0\cap A_1 \cap A_2 \cap A_3$. 
\end{enumerate}

It is possible for this equation form to have additional formal solutions that are not strong solutions. If we would like to find all of these additional solutions, then, according to Proposition~\ref{prop:root-difference0-prop-five}, solving the equation system 
\begin{equation*}
\systwo{\fx = \gx <0}{\hx =0,}
\end{equation*}
will result in all additional solutions that are formal solutions of the original radical equation but not strong solutions. 

We shall now state Proposition~\ref{prop:root-difference0-prop-four} and Proposition~\ref{prop:root-difference0-prop-five}. Proposition~\ref{prop:root-difference0-prop-five} justifies the solution procedure, described above, and states the set $S_1$ of all strong solutions and the set $S_2$ of all formal solutions. Proposition~\ref{prop:root-difference0-prop-four} is a mild generalization of Proposition~\ref{prop:one-radical} and is being used in the proof of Proposition~\ref{prop:root-difference0-prop-five}. Proposition~\ref{prop:root-difference0-prop-five} is illustrated with Example~\ref{ex:root-differences-example}, where we consider the equation $\sqrt{2a-x}-\sqrt{x-2b}=x-(a+b)$, under the restriction $a<b$, and show that it has no strong solutions, but does have $x=a+b$ as a formal solution. The case $a\geq b$ is an interesting challenge, and, for the sake of brevity, is left as an exercise to the reader. 

\begin{proposition}
Consider the equation $\hx\sqfx = \gx$ with $f: A\to\bbR$ and $g: B\to\bbR$ and $h: C\to\bbR$ polynomial or rational functions with $A\subseteq\bbR$ and $B\subseteq\bbR$ and $C\subseteq\bbR$. Then the set $S_1$ of all strong solutions is given by 
\begin{align*}
S_1 &= S_0 \cap (A_1\cup A_2), \\
S_0 &= \{ x\in A\cap B\cap C \;|\; [\hx]^2 \fx = [\gx]^2 \}, \\
A_1 &= \{  x\in A\cap B\cap C \;|\;  \gx\hx\geq 0 \land \hx\neq 0 \}, \\
A_2 &=  \{  x\in A\cap B\cap C \;|\;   \hx = 0 \land \fx\geq 0 \},
\end{align*}
and the set $S_2$ of all formal solutions is given by 
\begin{align*}
S_2 &= S_0 \cap A_3, \\
A_3 &= \{ x\in A\cap B\cap C \;|\; \gx\hx \geq 0 \}.
\end{align*}
\label{prop:root-difference0-prop-four}
\end{proposition}

\begin{proof}
\emph{Strong solutions:} We begin by determining the set $S_1$ of all strong solutions. Let $x\in A\cap B\cap C$ be given. All strong solutions must satisfy $\fx\geq 0$. We proceed by distinguishing between the following cases: 

\noindent
\emph{Case 1:} Assume that $\hx\neq 0$. Then, using Proposition~\ref{prop:one-radical}, we can solve for $x$ as follows: 
\begin{align*}
\hx\sqfx &= \gx
\ifonlyif
\sqfx = \frac{\gx}{\hx} && [\text{via } \hx\neq 0] \\
&\ifonlyif
\fx =  \fracp{\gx}{\hx}^2  \land \frac{\gx}{\hx}\geq 0   && [\text{via Proposition~\ref{prop:one-radical}}] \\
&\ifonlyif
[\hx]^2 \fx = [\gx]^2 \land \gx\hx\geq 0 \land \hx\neq 0 \\
&\ifonlyif
x\in S_0 \cap A_1.
\end{align*}
We note that since all solutions $x\in S_0 \cap A_1$ satisfy $\fx = [\gx/\hx]^2$, they also satisfy the strong solution requirement $\fx\geq 0$. 

\noindent
\emph{Case 2:} Assume that $\hx =0$. Then, it follows that 
\begin{align*}
\hx\sqfx = \gx
&\ifonlyif
0\sqfx = \gx
\ifonlyif
0 = \gx
\ifonlyif
0 = [\gx]^2 \\
&\ifonlyif
[\hx]^2 \fx = [\gx]^2.
\end{align*}
 Note that the condition $\fx\geq 0$ is now needed to ensure that we limit ourselves to strong solutions. 

We conclude that the set of all strong solutions is given by $S_1 = (S_0 \cap A_1) \cup (S_0 \cap A_2) = S_0 \cap (A_1 \cup A_2)$. 

\emph{Formal solutions:} The argument for deriving the set $S_2$ of all formal solutions is almost identical to the argument given above for finding the strong solutions. The only difference is that for Case 2 (i.e. under the assumption $\hx =0$), we remove the restriction $\fx\geq 0$. For Case 1 (i.e. under the assumption $\hx\neq 0$), the set of strong and formal solutions coincide, in accordance with Proposition~\ref{prop:one-radical}. It follows that the set of all formal solutions is given by $S_2 = S_0 \cap (B_1 \cup B_2)$ with 
\begin{align*}
B_1 &= \{ x\in A\cap B\cap C \;|\; \gx\hx\geq 0 \land \hx\neq 0 \}, \\
B_2 &= \{ x\in A\cap B\cap C \;|\; \hx  = 0 \}.
\end{align*}
Since $B_1 \cup B_2 = \{  x\in A\cap B\cap C \;|\;  \gx\hx\geq 0 \} = A_3$ we conclude that $S_2 = S_0 \cap A_3$. 
\end{proof}

\begin{proposition}
Consider the equation $\sqfx-\sqgx = \hx$ with $f: A\to\bbR$ and $g: B\to\bbR$ and $h: C\to\bbR$ polynomial or rational functions with $A\subseteq\bbR$ and $B\subseteq\bbR$ and $C\subseteq\bbR$. Then the set $S_1$ of all strong solutions is given by
\begin{align*}
S_1 &= S_0 \cap A_1 \cap A_2 \cap A_3, \\
S_0 &= \{ x\in A\cap B\cap C \;|\;  4[\hx]^2 \fx = [\fx+[\hx]^2-\gx]^2 \}, \\
A_1 &= \{ x\in A\cap B\cap C \;|\;  \fx\geq 0 \}, \\
A_2 &= \{ x\in A\cap B\cap C \;|\;  \sqfx-\hx\geq 0 \}, \\
A_3 &= \{ x\in A\cap B\cap C \;|\; \hx [\fx+[\hx]^2-\gx] \geq 0 \},
\end{align*}
 and the set $S_2$ of all formal solutions is given by 
\begin{align*}
S_2 &= (S_0 \cap A_1 \cap A_2 \cap A_3)\cup B_1, \\
B_1 &= \{ x\in A\cap B\cap C \;|\; \fx = \gx <0 \land \hx =0 \}.
\end{align*}
\label{prop:root-difference0-prop-five}
\end{proposition}

\begin{proof}
\emph{Strong solutions:} We begin with the argument that determines the set of all strong solutions. Let $ x\in A\cap B\cap C$ be given. All strong solutions must satisfy $\fx\geq 0$ and $\gx\geq 0$, however, only the condition $\fx\geq 0$ is needed to carry the argument forward, so let us assume just that the chosen $x$ satisfies $\fx\geq 0$. Then, using Proposition~\ref{prop:one-radical} twice, we have: 
\begin{align}
\sqfx - \sqgx = \hx 
&\ifonlyif 
\sqgx = \sqfx - \hx \nonumber \\
&\ifonlyif 
\systwo{\gx = [\sqfx-\hx]^2}{\sqfx-\hx\geq 0,}
\label{eq:root-difference0-prop-five-eq-one}
\end{align}
and 
\begin{align}
\gx = [\sqfx-\hx]^2
&\ifonlyif 
\gx = \fx-2\hx\sqfx+[\hx]^2 \nonumber \\
&\ifonlyif 
2\hx\sqfx = \fx-\gx+[\hx]^2 \nonumber \\
&\ifonlyif 
\systwo{4(\hx)^2 \fx = (\fx-\gx+(\hx)^2)^2}{\hx [\fx-\gx+(\hx)^2] \geq 0.}
\label{eq:root-difference0-prop-five-eq-two}
\end{align}
The first application of Proposition~\ref{prop:one-radical}  is justified because we are limiting ourselves to finding only strong solutions. The last equivalence leading to Eq.~\eqref{eq:root-difference0-prop-five-eq-two} follows from Proposition~\ref{prop:root-difference0-prop-four},  noting that, given the a priori assumption that $\fx\geq 0$, we do not need to consider separately the possibilities $\hx=0$ vs $\hx\neq 0$, corresponding to the sets $A_1, A_2$ in the statement of Proposition~\ref{prop:root-difference0-prop-four}, since their union then simplifies to the definition of $A_3$ in the statement of Proposition~\ref{prop:root-difference0-prop-four}. Combining Eq.~\eqref{eq:root-difference0-prop-five-eq-one} and Eq.~\eqref{eq:root-difference0-prop-five-eq-two} gives, under the assumption $\fx\geq 0$, that
\begin{align}
\sqfx-\sqgx = \hx
&\ifonlyif
\systhree{4(\hx)^2 \fx = [\fx-\gx+(\hx)^2)^2}{\hx [\fx-\gx+(\hx)^2] \geq 0}{\sqfx-\hx\geq 0} \nonumber \\
&\ifonlyif x\in S_1 = S_0 \cap A_1 \cap A_2 \cap A_3,
\label{eq:root-difference0-prop-five-eq-three}
\end{align}
which confirms our claim regarding the strong solutions set $S_1$. 

\emph{Formal solutions:} We now determine the solution set $S_2$ of all formal solutions. Let $x\in A\cap B\cap C$ be given, with no other restrictions imposed on $x$. We note that 
\begin{equation}
\sqfx-\sqgx=\hx \ifonlyif \sqgx = \sqfx-\hx.
\label{eq:root-difference0-prop-five-eq-four}
\end{equation}
We cannot apply Proposition~\ref{prop:one-radical} because the right-hand-side of Eq.~\eqref{eq:root-difference0-prop-five-eq-four} may be complex, if $\fx <0$. Therefore, in order to move the argument forward, we distinguish between the following cases:

\noindent
\emph{Case 1:} Assume that $\fx\geq 0$. Then Eq.~\eqref{eq:root-difference0-prop-five-eq-one}, Eq.~\eqref{eq:root-difference0-prop-five-eq-two}, and Eq.~\eqref{eq:root-difference0-prop-five-eq-three} can be derived yet again by the exact same argument used to find all strong solutions. As a result, this case contributes the formal solutions in the set $S_0 \cap A_1 \cap A_2 \cap A_3$.

\noindent
\emph{Case 2:} Assume that $\fx <0$ and $\gx\geq 0$. Then it follows that
\begin{align*}
\Im (\sqgx) &= \Im (\sqfx-\hx) && [\text{via the equation}] \\
&= \Im (\sqfx) && [\text{via } \hx\in\bbR] \\
&\neq 0, && [\text{via } \fx <0]
\end{align*}
 which is a contradiction, since $\gx\geq 0 \implies \Im (\sqgx) = 0$, consequently, this subcase does not contribute any additional solutions. 

\noindent
\emph{Case 3:} Assume that $\fx <0$ and $\gx <0$ and $\hx\neq 0$. Then it follows that, 
\begin{align*}
\Re (\sqgx) &= \Re (\sqfx-\hx) && [\text{via the equation}] \\
&= \Re (-\hx) && [\text{via } \fx <0] \\
&\neq 0, && [\text{via } \hx\neq 0]
\end{align*}
which is a contradiction, since $\gx<0 \implies \Re (\sqgx) =0$, consequently this subcase does not contribute any additional solutions either. 

\noindent
\emph{Case 4:} Assume that $\fx <0$ and $\gx <0$ and $\hx =0$. Then it follows that 
\begin{align*}
\sqgx &= \sqfx-\hx
\ifonlyif
\sqgx = \sqfx && [\text{via } \hx =0] \\
&\ifonlyif
i\sqrt{-\gx} = i\sqrt{-\fx} \\
&\ifonlyif
\sqrt{-\gx} = \sqrt{-\fx} \\
&\ifonlyif
-\gx = -\fx &&[\text{via } \fx <0 \land \gx <0] \\
&\ifonlyif
\gx = \fx.
\end{align*}
This results in possible additional formal solutions given by the set 
\begin{equation*}
B_1 = \{ x\in A\cap B\cap C \;|\; \fx = \gx <0 \land \hx =0 \}.
\end{equation*}
From the above argument, it follows that the set of all formal solutions is given by $S_2 = (S_0\cap A_1 \cap A_3 \cap A_4)\cup B_1$. 
\end{proof}

\begin{example}
The equation $\sqrt{2a-x}-\sqrt{x-2b}=x-(a+b)$ with $a<b$, has no strong solutions and has $x=a+b$ as a formal, but not strong, solution.
\label{ex:root-differences-example}
\end{example}

\begin{solution}
We employ Proposition~\ref{prop:root-difference0-prop-five} using $\fx=2a-x$, $\gx=x-2b$, and $\hx=x-(a+b)$, all defined on $\bbR$, and under the assumption $a<b$. The restriction set $A_1$ is given by
\begin{equation*}
A_1=\{x\in\bbR \;|\; \fx\geq 0\}=\{x\in\bbR\;|\; 2a-x\geq 0\}=(-\infty, 2a].
\end{equation*}
For the restriction set $A_2$, we note that 
\begin{equation*}
x\in A_2 \ifonlyif \sqfx-\hx\geq 0 \ifonlyif \sqrt{2a-x}-[x-(a+b)]\geq 0.
\end{equation*}
To solve the corresponding inequality, we define
\begin{equation*}
\gf (x)=\sqrt{2a-x}-[x-(a+b)],\;\forall x\in A_1,
\end{equation*}
and note that $\gf (2a)=b-a>0$ and that the derivative of $\gf (x)$ satisfies
\begin{equation*}
\gf'(x)=-\frac{2\sqrt{2a-x}+1}{2\sqrt{2a-x}}<0,\;\forall x\in (-\infty, 2a).
\end{equation*}
It follows that $\forall x\in (-\infty, 2a) : \gf (x)>\gf (2a)>0$, and therefore $(-\infty, 2a]\subseteq A_2$, from which it follows that $A_1\cap A_2=(-\infty, 2a]$. For the restriction set $A_3$, we note that
\begin{align*}
x\in A_3 &\ifonlyif \hx [\fx+(\hx)^2-\gx]\geq 0 \\
&\ifonlyif [x-(a+b)][(2a-x)+(x-(a+b))^2-(x-2b)]\geq 0 \\
&\ifonlyif [x-(a+b)]^2 (x-2-(a+b))\geq 0 \\
&\ifonlyif x-(a+b)=0\lor x-2-(a+b)\geq 0 \\
&\ifonlyif x=a+b \lor x\geq 2+(a+b),
\end{align*}
from which it follows that
\begin{equation*}
A_3 = \{a+b\}\cup[2+(a+b),+\infty).
\end{equation*}
To find the intersection $A_1\cap A_2\cap A_3$ we note that the assumption $a<b$ implies that $a+b>a+a=2a \implies a+b\not\in (-\infty, 2a]$ and $2+(a+b)>2+(a+a)>2a$, from which it follows that
\begin{align*}
A_1\cap A_2 \cap A_3 &= (-\infty, 2a]\cap [\{a+b\}\cup [2+(a+b),+\infty)] \\
&= (-\infty, 2a]\cap [2+(a+b), +\infty) = \emptyset.
\end{align*}
We conclude that the equation does not have any strong solutions. To check whether the set $B_1$ contains any additional formal solutions, we note that
\begin{align*}
x\in B_1 &\ifonlyif \systwo{\fx=\gx<0}{\hx=0} \ifonlyif \systwo{2a-x=x-2b<0}{x-(a+b)=0} \\ 
&\ifonlyif \systwo{x=a+b}{(a+b)-2b<0} \ifonlyif x=a+b,
\end{align*}
noting that $a<b$ implies that $(a+b)-2b<0$. We conclude that $x=a+b$ is a formal, but not strong, solution of the original equation.
\end{solution}

\section{Conclusion}

The main result of this article is Proposition~\ref{prop:sum-two-roots-equal-to-function}, Proposition~\ref{prop:sum-of-two-roots-equal-root}, and Proposition~\ref{prop:root-difference0-prop-five} that define the set of strong solutions and the wider set of formal solutions for radical equations with depth 2 that follow the form of Eq.~\eqref{eq:radical-form-four}, Eq.~\eqref{eq:radical-form-five}, Eq.~\eqref{eq:radical-form-six} correspondingly. Similarly to the propositions in our previous article \cite{article:Gkioulekas:radicals-one}, it is worth emphasizing  that the requirement that the functions $f,g,h$ in these propositions be polynomial or rational functions is needed only to justify the results about the solution set of all formal solutions,  and can be removed if we are only interested in the set of all strong solutions, in which case our propositions can be used recursively to tackle a few additional radical equation forms with depth above 2. We have also explained how these propositions translate into rigorous solution procedures and illustrated the solution procedures with solved examples. The proposed procedures allow the elimination of extraneous solutions via inequality restrictions, thereby making it unnecessary to verify each candidate solution explicitly against the original equation. As can be seen from the given examples, this is particularly useful with respect to handling parametric radical equations, or for developing theorems for particular forms of radical equations.

The reported results, combined with the preceding results \cite{article:Gkioulekas:radicals-one}, can be incorporated in mathematics coursework in many ways. For lower-level coursework, students can be taught the informal solution techniques, as described in the beginning of Section 2, Section 3, and Section 4, and use them primarily on non-parametric examples. Instructors have a choice on whether the scope of their teaching should be limited to strong solutions, or whether it should include finding formal solutions as well. In either case, for students that have  already been taught the concept of complex numbers, it is straightforward to at least teach them about the distinction between the two types of solutions. More advanced students can be exposed to the more rigorous presentation of Proposition~\ref{prop:sum-two-roots-equal-to-function}, Proposition~\ref{prop:sum-of-two-roots-equal-root}, Proposition~\ref{prop:root-difference0-prop-five}, and their proofs. The proofs themselves are nice examples of proof-writing, illustrating proof by contradiction, proof by cases, and how to organize a complex proof, all illustrated in the context of an elementary area of mathematics. We provided detailed proofs in order to also demonstrate how the underlying proof-writing technique can be presented to students. Finally, the theorems themselves can serve as a foundation for undergraduate capstone student projects, where they can be used to deeply explore a very wide range of parametric radical equations. The three parametric examples presented in this article are only the tip of the iceberg. 

\section*{Acknowledgements}

I was first introduced to rigorous solution methods for solving radical equations by my high-school teacher, Mr. Alexandros Pistofidis. 

\section*{Disclosure statement}

No potential conflict of interest was reported by the author.

\section*{ORCID}

\emph{Eleftherios Gkioulekas}: \textsf{http://orcid.org/0000-0002-5437-2534}

\bibliographystyle{tfnlm}
\bibliography{references-mathed}

\appendix

\section{Position of numbers relative to the two zeroes of a quadratic}
\label{app:numbers-relative-to-zeroes}

Consider a quadratic function $\fx= ax^2+bx+c$, where we assume, with no loss of generality, that $a>0$, with two real zeroes $x_1, x_2\in\bbR$ such that $x_1 < x_2$. It is well known that $\fx$ will be negative over the interval $(x_1, x_2)$ and positive over the set $(-\infty, x_1)\cup (x_2, +\infty)$. This simple observation can be used to determine the position of arbitrary real numbers $\xi_1, \xi_2\in\bbR$ with $\xi_1 < \xi_2$ relative to the zeroes $x_1, x_2$, using an indirect argument, instead of attempting to directly prove the corresponding inequalities. This is a technique that I first learned from \cite{notes:Pistofidis:1989}, and, as can be seen from Example~\ref{ex:sum-of-two-roots-equal-root} and Example~\ref{ex:root-differences-example}, it can be most useful in testing candidate solutions against the restriction sets, to establish whether the candidate solutions can be accepted or rejected. 

To begin with, it is fairly obvious that
\begin{equation*}
f(\xi_1)<0 \ifonlyif x_1 < \xi_1 < x_2,
\end{equation*}
since the quadratic is negative only between the two zeroes $x_1, x_2$. We can also argue that
\begin{align}
f(\xi_1)<0 \land f(\xi_2)>0 &\ifonlyif x_1<\xi_1<x_2<\xi_2, \label{eq:numbers-relative-to-zeroes-eq-one} \\
f(\xi_1)>0 \land f(\xi_2)<0 &\ifonlyif \xi_1 < x_1 < \xi_2 < x_2. \label{eq:numbers-relative-to-zeroes-eq-two}
\end{align}
To justify Eq.~\eqref{eq:numbers-relative-to-zeroes-eq-one}, we note that $\xi_1$ is between $x_1, x_2$ and $\xi_2$ is not. To determine that $\xi_2\in (x_2, +\infty)$, we take advantage of the assumption $\xi_1<\xi_2$. Eq.~\eqref{eq:numbers-relative-to-zeroes-eq-two} is justified with a similar argument, and Eq.~\eqref{eq:numbers-relative-to-zeroes-eq-one} and Eq.~\eqref{eq:numbers-relative-to-zeroes-eq-two} combined have been sufficient in handling Example~\ref{ex:sum-of-two-roots-equal-root} and Example~\ref{ex:root-differences-example}. 

One additional result, that can be useful for other problems, is that
\begin{equation*}
f(\xi_1)<0 \land f(\xi_2)<0 \ifonlyif x_1<\xi_1<\xi_2<x_2.
\end{equation*}
Last, but not least, when both $f(\xi_1)$ and $f(\xi_2)$ are positive, we have to use instead the following statements:
\begin{align*}
f(\xi_1)>0 \land \xi_1+\frac{b}{2a}<0 &\ifonlyif \xi_1<x_1<x_2, \\
f(\xi_1)>0 \land \xi_1+\frac{b}{2a}>0 &\ifonlyif x_1<x_2<\xi_1,
\end{align*}
where, we determine whether $\xi_1$ lies to the left or to the right of the zeroes $x_1, x_2$ by comparing its location relative to the $x$-coordinate of the quadratic's vertex. 

\end{document}